\documentclass[12pt]{article}
\usepackage[]{amsmath,amssymb}
\usepackage{amscd}
\usepackage{amsthm}
\usepackage{latexsym}
\usepackage{cite}
\usepackage[table]{xcolor}
\usepackage{pifont}
\usepackage[all]{xy}

\usepackage{hyperref}

\newtheorem{definition}{Definition}[section]
\newtheorem{theorem}[definition]{Theorem}
\newtheorem{lemma}[definition]{Lemma}
\newtheorem{corollary}[definition]{Corollary}

\newtheorem{note}[definition]{Note}

\newtheorem{notation}[definition]{Notation}

\def\N{\mathbb N}

\def\Z{\mathbb Z}
\def\K{\mathbb K}
\def\A{\mathcal A}

\renewcommand{\tilde}{\widetilde}
\renewcommand{\epsilon}{\varepsilon}

\numberwithin{equation}{section}

\newcommand{\cmark}{\ding{51}}%
\newcommand{\xmark}{\ding{55}}%

\begin{document}
\title{\bf
The positive even subalgebra of $U_q(\mathfrak{sl}_2)$ and its finite-dimensional irreducible modules}
\author{
Alison Gordon Lynch}
\date{}

\maketitle
\begin{abstract}
The equitable presentation of $U_q(\mathfrak{sl}_2)$ was introduced in 2006 by Ito, Terwilliger, and Weng. This presentation involves some generators $x, y, y^{-1}, z$. It is known that $\{x^r y^s z^t : r, t \in \mathbb{N}, s \in \mathbb{Z}\}$ is a basis for the $\K$-vector space $U_q(\mathfrak{sl}_2)$.  In 2013, Bockting-Conrad and Terwilliger introduced a subalgebra $\mathcal{A}$ of $U_q(\mathfrak{sl}_2)$ spanned by the elements $\{x^r y^s z^t : r, s, t \in \mathbb{N}, r+s+t \ {\rm even}\}$.  We give a presentation of $\A$ by generators and relations. We also classify up to isomorphism the finite-dimensional irreducible $\A$-modules, under the assumption that $q$ is not a root of unity.

 \end{abstract}

\section{Introduction}
Throughout this paper, let $\K$ denote a field and let $q$ denote a nonzero scalar in $\K$ such that $q^2 \ne 1$.
\medskip
We recall the quantum algebra $U_q(\mathfrak{sl}_2)$.  We will use the equitable presentation for $U_q(\mathfrak{sl}_2)$, which was introduced in \cite{Ito2006}. By \cite[Theorem 2.1]{Ito2006}, the equitable presentation of $U_q(\mathfrak{sl}_2)$ has generators $x, y^{\pm 1}, z$ and relations $y y^{-1} = 1, y^{-1} y = 1$,
\begin{equation*}
    \frac{qxy-q^{-1}yx}{q-q^{-1}} = 1, \qquad \frac{qyz-q^{-1}zy}{q-q^{-1}} = 1, \qquad \frac{qzx-q^{-1}xz}{q-q^{-1}} = 1.
\end{equation*}

\medskip
 The equitable presentation for $U_q(\mathfrak{sl}_2)$ has connections with Leonard pairs \cite{Alnajjar2011}, Leonard triples \cite{Gao2013},\cite{Huang2012}, tridiagonal pairs \cite{Bockting-Conrad2014}, bidiagonal pairs \cite{Funk-Neubauer2013}, the $q$-tetrahedron algebra \cite{Ito2014},\cite{Ito2007}, the universal Askey-Wilson algebra \cite{Terwilliger2011}, Poisson algebras \cite{Jordan2010}, billiard arrays \cite{Terwilliger2014}, and distance-regular graphs \cite{Worawannotai2013}.

\medskip
By \cite[Lemma 10.7]{Terwilliger2011}, the $\K$-vector space $U_q(\mathfrak{sl}_2)$ has a basis
\begin{equation*}
    x^r y^s z^t \qquad r, t \in \N, \ s \in \Z.
\end{equation*}

We consider the subalgebra of $U_q(\mathfrak{sl}_2)$ spanned by the elements
\begin{equation*}
    x^r y^s z^t \qquad r, s, t \in \N, \ r+s+t \text{ even}.
\end{equation*}
This subalgebra was first discussed in \cite{Bockting-Conrad2013}. We call this subalgebra the {\it positive even subalgebra} of $U_q(\mathfrak{sl}_2)$, and denote it by $\A$.

\medskip
In this paper, we obtain two main results. In the first result, we give a presentation for $\A$ by generators and relations. In the second result, we classify up to isomorphism the finite-dimensional irreducible $\A$-modules, under the assumption that $q$ is not a root of unity.

\medskip
We now describe our first result in detail. We consider the elements $\nu_x, \nu_y, \nu_z$ of $U_q(\mathfrak{sl}_2)$, defined by
\begin{align*}
\nu_x = q(1-yz), \qquad \nu_y = q(1-zx), \qquad \nu_z = q(1-xy).
\end{align*}
By \cite[Proposition 5.4]{Bockting-Conrad2013}, the elements $\nu_x, \nu_y, \nu_z$ generate $\A$.
We also consider the elements $x^2, y^2, z^2$ of $U_q(\mathfrak{sl}_2)$. We show that $x^2, y^2, z^2$ generate $\A$ provided that $q^4 \ne 1$. We obtain some relations between $\nu_x, \nu_y, \nu_z$ and $x^2, y^2, z^2$. We also obtain some relations (\ref{eq:defrel1a})--(\ref{eq:defrel2f}) satisfied by $\nu_x, \nu_y, \nu_z$. We show that the unital associative $\K$-algebra with generators $\nu_x, \nu_y, \nu_z$ and relations (\ref{eq:defrel1a})--(\ref{eq:defrel2f}) is isomorphic to $\A$.

\medskip
We now describe our second result in detail. Assume that $q$ is not a root of unity. By \cite[Theorem 2.6]{Jantzen1996}, there exists a family of finite-dimensional irreducible $U_q(\mathfrak{sl}_2)$-modules
\begin{equation*}
    L(d, \epsilon), \qquad \epsilon \in \{1, -1\}, \ d \in \N.
\end{equation*}
For char $\K = 2$ we interpret the set $\{1, -1\}$ as $\{1\}$. By \cite[Theorem 2.6]{Jantzen1996}, every finite-dimensional irreducible $U_q(\mathfrak{sl}_2)$-module is isomorphic to exactly one of the modules $L(d, \epsilon)$. Let $d \in \N$. By restricting from $U_q(\mathfrak{sl}_2)$ to $\A$, each of $L(d,1)$ and $L(d,-1)$ becomes an $\A$-module. We show that the $\A$-modules $L(d,1)$ and $L(d,-1)$ are isomorphic and we denote the resulting $\A$-module by $L(d)$. We show that the $\A$-module $L(d)$ is irreducible.
We also show that each irreducible $\A$-module with dimension $d+1$ is isomorphic to $L(d)$.
Thus, we show that, up to isomorphism, $L(d)$ is the unique irreducible $\A$-module of dimension $d+1$.

\section{The quantum algebra \texorpdfstring{$U_q(\mathfrak{sl}_2)$}{Uq(sl2)}}

In this section, we recall the quantum algebra $U_q(\mathfrak{sl}_2)$ and its equitable presentation.

\begin{definition}\rm
\cite[Definition 2.2]{Ito2006}
Let $U_q(\mathfrak{sl}_2)$ denote the unital associative $\K$-algebra with generators $x, y^{\pm 1}, z$ and the following relations:
\begin{align}
y y^{-1} = y^{-1}y &= 1,\\
\frac{qxy-q^{-1}yx}{q-q^{-1}} &= 1,\label{eq:equit1}\\
\frac{qyz-q^{-1}zy}{q-q^{-1}} &= 1,\label{eq:equit2}\\
\frac{qzx-q^{-1}xz}{q-q^{-1}} &= 1. \label{eq:equit3}
\end{align}
We call $x, y^{\pm 1}, z$ the {\it equitable generators} for $U_q(\mathfrak{sl}_2)$.
\end{definition}

\begin{lemma}\label{lem:reform}
The following relations hold in $U_q(\mathfrak{sl}_2)$:
\begin{align}
xy &= q^{-2}yx - q^{-2}+1, \qquad yx = q^2 xy - q^2 +1,\label{eq:flip1}\\
yz &= q^{-2}zy - q^{-2}+1, \qquad zy = q^2 yz - q^2 +1,\label{eq:flip2}\\
zx &= q^{-2}xz - q^{-2}+1, \qquad xz = q^2 zx - q^2 +1\label{eq:flip3}.
\end{align}
\end{lemma}
\begin{proof}
The equations (\ref{eq:flip1})--(\ref{eq:flip3}) are reformulations of (\ref{eq:equit1})--(\ref{eq:equit3}).
\end{proof}

Recall the natural numbers $\N = \{0, 1, 2, \ldots\}$ and the integers $\Z = \{0, \pm 1, \pm 2, \ldots\}$.
\begin{lemma}{\rm \cite[Lemma 10.7]{Terwilliger2011}}\label{lem:Uqsl2basis} The following is a basis for the $\K$-vector space $U_q(\mathfrak{sl}_2)$:
\begin{equation}\label{eq:Uqsl2basis}
    x^r y^s z^t \qquad r, t \in \N, \ s \in \Z.
\end{equation}
\end{lemma}

\section{The elements \texorpdfstring{$\nu_x, \nu_y, \nu_z$}{nu x, nu y, nu z}}\label{sec:nus}
In this section, we recall the elements $\nu_x, \nu_y, \nu_z$ of $U_q(\mathfrak{sl}_2)$ and we give relations involving $\nu_x, \nu_y, \nu_z$ and $x^2, y^2, z^2$.

\medskip
The relations (\ref{eq:equit1})--(\ref{eq:equit3}) can be reformulated as follows:
\[
q(1-yz) = q^{-1}(1-zy), \qquad q(1-zx) = q^{-1}(1-xz), \qquad q(1-xy) = q^{-1}(1-yx).
\]
\begin{definition}{\rm \cite[Definition 3.1]{Terwilliger2011}}\rm \label{def:nu}
Let $\nu_x, \nu_y, \nu_z$ denote the following elements of $U_q(\mathfrak{sl}_2)$:
\begin{align}
\nu_x &= q(1-yz) = q^{-1}(1-zy),\label{eq:nux}\\
\nu_y &= q(1-zx) = q^{-1}(1-xz),\label{eq:nuy}\\
\nu_z &= q(1-xy) = q^{-1}(1-yx). \label{eq:nuz}
\end{align}
\end{definition}

\begin{lemma}{\rm \cite[Lemma 3.3]{Terwilliger2011}}\label{lem:rel1}
The following relations hold in $U_q(\mathfrak{sl}_2)$:
\begin{align}
xy &= 1 - q^{-1}\nu_z, \qquad yx = 1-q\nu_z,\label{eq:pairs1}\\
yz &= 1 - q^{-1}\nu_x, \qquad zy = 1-q\nu_x,\label{eq:pairs2}\\
zx &= 1 - q^{-1}\nu_y, \qquad xz = 1-q\nu_y.\label{eq:pairs3}
\end{align}
\end{lemma}
\begin{proof}
These equations are reformulations of (\ref{eq:nux})--(\ref{eq:nuz}).
\end{proof}

\begin{lemma}{\rm \cite[Lemma 3.5]{Terwilliger2011}}\label{lem:rel2a}
The following relations hold in $U_q(\mathfrak{sl}_2)$:
\begin{align*}
x \nu_y &= q^2 \nu_y x, \qquad x^2\nu_z = q^{-2}\nu_z x, \\
y \nu_z &= q^2 \nu_z y, \qquad y^2\nu_x = q^{-2}\nu_x y, \\
z \nu_x &= q^2 \nu_x z, \qquad z^2\nu_y = q^{-2}\nu_y z.
\end{align*}
\end{lemma}

In the next few lemmas, we give some relations between $\nu_x, \nu_y, \nu_z$ and $x^2, y^2, z^2$.

\begin{lemma}\label{lem:rel2}
The following relations hold in $U_q(\mathfrak{sl}_2)$:
\begin{align}
x^2 \nu_y &= q^4 \nu_y x^2, \qquad x^2\nu_z = q^{-4}\nu_z x^2, \label{eq:rel2a}\\
y^2 \nu_z &= q^4 \nu_z y^2, \qquad y^2\nu_x = q^{-4}\nu_x y^2, \label{eq:rel2b}\\
z^2 \nu_x &= q^4 \nu_x z^2, \qquad z^2\nu_y = q^{-4}\nu_y z^2. \label{eq:rel2c}
\end{align}
\end{lemma}
\begin{proof}
Immediate from Lemma \ref{lem:rel2a}.
\end{proof}

\begin{lemma}\label{lem:rel3}
The following relations hold in $U_q(\mathfrak{sl}_2)$:
\begin{align}
x^2y^2 &= 1-q^{-2}(q+q^{-1})\nu_z + q^{-4}\nu_z^2, \label{eq:squareda}\\
y^2z^2 &= 1-q^{-2}(q+q^{-1})\nu_x + q^{-4}\nu_x^2,\label{eq:squaredb}\\
z^2x^2 &= 1-q^{-2}(q+q^{-1})\nu_y + q^{-4}\nu_y^2,\label{eq:squaredc}
\end{align}
and
\begin{align}
y^2x^2 &= 1-q^2(q+q^{-1})\nu_z + q^4\nu_z^2, \label{eq:squaredd}\\
z^2y^2 &= 1-q^2(q+q^{-1})\nu_x + q^4\nu_x^2, \label{eq:squarede}\\
x^2z^2 &= 1-q^2(q+q^{-1})\nu_y + q^4\nu_y^2. \label{eq:squaredf}
\end{align}
\end{lemma}
\begin{proof}
To verify (\ref{eq:squareda}), we compute $(xy)^2$ in two ways. Using Lemma \ref{lem:rel1},
\begin{equation}
(xy)^2 = (1 - q^{-1}\nu_z)^2 = 1-2q^{-1}\nu_z + q^{-2}\nu_z^2. \label{eq:ver1}
\end{equation}
 Using Lemma \ref{lem:reform} and Lemma \ref{lem:rel1},
\begin{align}
    (xy)^2 &= x(yx)y\notag\\
            &= x(q^2xy - q^2 + 1)y\notag\\
            &= q^2 x^2 y^2 + (1-q^2)xy\notag\\
            &= q^2 x^2 y^2 + (1-q^2)(1-q^{-1}\nu_z).\label{eq:ver2}
\end{align}
By equating the right-hand sides of (\ref{eq:ver1}) and (\ref{eq:ver2}), we obtain (\ref{eq:squareda}). The remaining relations are similarly obtained.
\end{proof}

\begin{lemma}\label{lem:rel4}
The following relations hold in $U_q(\mathfrak{sl}_2)$:
\begin{align}
x^2 \nu_x &= q^{-1} x^2 - q^{-1} + q^2 \nu_y + q^{-2}\nu_z - q\nu_y\nu_z,\label{eq:x2nuxa}\\
y^2 \nu_y &= q^{-1} y^2 - q^{-1} + q^2 \nu_z + q^{-2}\nu_x - q\nu_z\nu_x,\label{eq:x2nuxb}\\
z^2 \nu_z &= q^{-1} z^2 - q^{-1} + q^2 \nu_x + q^{-2}\nu_y - q\nu_x\nu_y,\label{eq:x2nuxc}
\end{align}
and
\begin{align}
\nu_x x^2 &= q^{-1}x^2 - q^{-1}+q^{-2}\nu_y + q^2\nu_z - q\nu_y\nu_z, \label{eq:x2nuxd}\\
\nu_y y^2 &= q^{-1}y^2 - q^{-1}+q^{-2}\nu_z + q^2\nu_x - q\nu_z\nu_x, \label{eq:x2nuxe}\\
\nu_z z^2 &= q^{-1}z^2 - q^{-1}+q^{-2}\nu_x + q^2\nu_y - q\nu_x\nu_y. \label{eq:x2nuxf}
\end{align}
\end{lemma}
\begin{proof}

We verify (\ref{eq:x2nuxa}). Using $\nu_x = q^{-1}(1-zy)$, we find
\[
x^2 \nu_x = q^{-2} x^2 - q^{-1} x^2 zy.
\]
By Lemma \ref{lem:reform} and Lemma \ref{lem:rel1},
\begin{align*}
x^2 z y &= x (q^2zx - q^2 + 1)y = q^2 (xz)(xy) + (1-q^2)xy\\
    &= q^2 (1-q\nu_y)(1-q^{-1}\nu_z) + (1-q^2)(1-q^{-1}\nu_z)\\
    &= 1 - q^3\nu_y - q^{-1}\nu_z +q^2\nu_y\nu_z.
\end{align*}

The equation (\ref{eq:x2nuxa}) follows from these comments. The remaining relations are similarly obtained.
\end{proof}

\begin{lemma}{\rm \cite[Lemma 3.10]{Terwilliger2011}}\label{lem:x2innu}
The following relations hold in $U_q(\mathfrak{sl}_2)$:
\begin{align}
x^2 &= 1 - \frac{q\nu_y\nu_z - q^{-1}\nu_z\nu_y}{q-q^{-1}},\label{eq:x2innua}\\
y^2 &= 1 - \frac{q\nu_z\nu_x - q^{-1}\nu_x\nu_z}{q-q^{-1}},\label{eq:x2innub}\\
z^2 &= 1 - \frac{q\nu_x\nu_y - q^{-1}\nu_y\nu_x}{q-q^{-1}}.\label{eq:x2innuc}
\end{align}
\end{lemma}

\begin{lemma}\label{lem:nuinx2}
Assume that $q^4 \ne 1$. Then the following relations hold in $U_q(\mathfrak{sl}_2)$:
\begin{align}
	(q+q^{-1})\nu_x &= q^2 + q^{-2} - \frac{q^4 y^2 z^2 - q^{-4}z^2y^2}{q^2-q^{-2}},\label{eq:nuinx2a}\\
	(q+q^{-1})\nu_y &= q^2 + q^{-2} - \frac{q^4 z^2 x^2 - q^{-4}x^2z^2}{q^2-q^{-2}},\label{eq:nuinx2b}\\
	(q+q^{-1})\nu_z &= q^2 + q^{-2} - \frac{q^4 x^2 y^2 - q^{-4}y^2x^2}{q^2-q^{-2}}.\label{eq:nuinx2c}
\end{align}
\end{lemma}
\begin{proof}
For each equation, evaluate the right-hand side using Lemma \ref{lem:rel3}.
\end{proof}

We now display some relations satisfied by $\nu_x, \nu_y, \nu_z$ in $U_q(\mathfrak{sl}_2)$.
\begin{lemma}\label{lem:defrel1}
The following relations hold in $U_q(\mathfrak{sl}_2)$:
\begin{align}
q^3 \nu_x^2\nu_y-(q+q^{-1})\nu_x\nu_y\nu_x + q^{-3}\nu_y\nu_x^2 &= (q^2-q^{-2})(q-q^{-1})\nu_x,\label{eq:defrel1a}\\
q^3 \nu_y^2\nu_z-(q+q^{-1})\nu_y\nu_z\nu_y + q^{-3}\nu_z\nu_y^2 &= (q^2-q^{-2})(q-q^{-1})\nu_y,\label{eq:defrel1b}\\
q^3 \nu_z^2\nu_x-(q+q^{-1})\nu_z\nu_x\nu_z + q^{-3}\nu_x\nu_z^2 &= (q^2-q^{-2})(q-q^{-1})\nu_z,\label{eq:defrel1c}
\end{align}
and
\begin{align}
q^{-3} \nu_y^2\nu_x-(q+q^{-1})\nu_y\nu_x\nu_y + q^{3}\nu_x\nu_y^2 &= (q^2-q^{-2})(q-q^{-1})\nu_y,\label{eq:defrel1d}\\
q^{-3} \nu_z^2\nu_y-(q+q^{-1})\nu_z\nu_y\nu_z + q^{3}\nu_y\nu_z^2 &= (q^2-q^{-2})(q-q^{-1})\nu_z,\label{eq:defrel1e}\\
q^{-3} \nu_x^2\nu_z-(q+q^{-1})\nu_x\nu_z\nu_x + q^{3}\nu_z\nu_x^2 &= (q^2-q^{-2})(q-q^{-1})\nu_x.\label{eq:defrel1f}
\end{align}
\end{lemma}
\begin{proof}
Observe that (\ref{eq:defrel1a}) is equivalent to
\begin{equation}\label{eq:defrel1}
q^2 \nu_x \left(\frac{q\nu_x\nu_y - q^{-1}\nu_y\nu_x}{q-q^{-1}}\right) - q^{-2}\left(\frac{q\nu_x\nu_y - q^{-1}\nu_y\nu_x}{q-q^{-1}}\right)\nu_x = (q^2-q^{-2})\nu_x.
\end{equation}

To verify (\ref{eq:defrel1}), simplify the left-hand side  using Lemma \ref{lem:x2innu} and Lemma \ref{lem:rel2}. The remaining relations are similarly obtained.
\end{proof}

\begin{lemma}\label{lem:defrel2}
The following relations hold in $U_q(\mathfrak{sl}_2)$:
\begin{align}
\nu_x \frac{q\nu_y\nu_z - q^{-1}\nu_z\nu_y}{q-q^{-1}} &= \nu_x - q^{-2}\nu_y - q^2\nu_z + \frac{q^2\nu_y\nu_z-q^{-2}\nu_z\nu_y}{q-q^{-1}},\label{eq:defrel2a}\\
\nu_y \frac{q\nu_z\nu_x - q^{-1}\nu_x\nu_z}{q-q^{-1}} &= \nu_y - q^{-2}\nu_z - q^2\nu_x + \frac{q^2\nu_z\nu_x-q^{-2}\nu_x\nu_z}{q-q^{-1}},\label{eq:defrel2b}\\
\nu_z \frac{q\nu_x\nu_y - q^{-1}\nu_y\nu_x}{q-q^{-1}} &= \nu_z - q^{-2}\nu_x - q^2\nu_y + \frac{q^2\nu_x\nu_y-q^{-2}\nu_y\nu_x}{q-q^{-1}},\label{eq:defrel2c}
\end{align}
and
\begin{align}
\frac{q\nu_y\nu_z - q^{-1}\nu_z\nu_y}{q-q^{-1}}\nu_x &= \nu_x - q^{2}\nu_y - q^{-2}\nu_z + \frac{q^2\nu_y\nu_z-q^{-2}\nu_z\nu_y}{q-q^{-1}},\label{eq:defrel2d}\\
\frac{q\nu_z\nu_x - q^{-1}\nu_x\nu_z}{q-q^{-1}}\nu_y &= \nu_y - q^{2}\nu_z - q^{-2}\nu_x + \frac{q^2\nu_z\nu_x-q^{-2}\nu_x\nu_z}{q-q^{-1}},\label{eq:defrel2e}\\
\frac{q\nu_x\nu_y - q^{-1}\nu_y\nu_x}{q-q^{-1}}\nu_z &= \nu_z - q^{2}\nu_x - q^{-2}\nu_y + \frac{q^2\nu_x\nu_y-q^{-2}\nu_y\nu_x}{q-q^{-1}}.\label{eq:defrel2f}
\end{align}
\end{lemma}
\begin{proof}
We verify (\ref{eq:defrel2a}). By Lemma \ref{lem:x2innu} and Lemma \ref{lem:rel4},
\begin{align*}
\nu_x \frac{q\nu_y\nu_z - q^{-1}\nu_z\nu_y}{q-q^{-1}} &= \nu_x - \nu_x x^2\\
    &= \nu_x - q^{-2}\nu_y - q^2\nu_z + q^{-1}(1-x^2) +q\nu_y\nu_z\\
    &= \nu_x - q^{-2}\nu_y - q^2\nu_z + q^{-1}\frac{q\nu_y\nu_z - q^{-1}\nu_z\nu_y}{q-q^{-1}} + q\nu_y\nu_z\\
    &= \nu_x - q^{-2}\nu_y - q^2\nu_z + \frac{q^2\nu_y\nu_z - q^{-2}\nu_z\nu_y}{q-q^{-1}}.
\end{align*}

The remaining relations are similarly obtained.

\end{proof}

\section{A subalgebra of \texorpdfstring{$U_q(\mathfrak{sl}_2)$}{Uq(sl2)}}

In the previous section, we considered the elements $\nu_x, \nu_y, \nu_z$ of $U_q(\mathfrak{sl}_2)$. As we will see, $\nu_x, \nu_y, \nu_z$ do not generate all of $U_q(\mathfrak{sl}_2)$. We will be concerned with the subalgebra that they do generate. In this section, we define a $\K$-subspace $\A$ of $U_q(\mathfrak{sl}_2)$. We show that $\A$ is the subalgebra of $U_q(\mathfrak{sl}_2)$ generated by $\nu_x, \nu_y, \nu_z$. We also give two bases for the $\K$-vector space $\A$.

\begin{definition}{\rm \cite[Definition 4.5]{Bockting-Conrad2013}}\rm \label{def:A}
Let $\A$ denote the $\K$-subspace of $U_q(\mathfrak{sl}_2)$ spanned by the elements
\begin{equation}\label{eq:Abasis}
    x^r y^s z^t \qquad r, s, t \in \N, \ r+s+t \text{ even}.
\end{equation}
\end{definition}

\begin{lemma}{\rm\cite[Lemma 4.6]{Bockting-Conrad2013}}
$\A$ is a $\K$-subalgebra of $U_q(\mathfrak{sl}_2)$.
\end{lemma}

There are two generating sets of interest to us for the $\K$-algebra $\A$. The first generating set was given in \cite{Bockting-Conrad2013}. We include a proof for completeness.

\begin{lemma}{\rm \cite[Proposition 5.4]{Bockting-Conrad2013}}\label{lem:nugenset}
The $\K$-algebra $\A$ is generated by $\nu_x, \nu_y, \nu_z$.
\end{lemma}
\begin{proof}
Let $W$ denote the subalgebra of $U_q(\mathfrak{sl}_2)$ generated by $\nu_x, \nu_y, \nu_z$. We show that $W = \A$. Certainly $W \subseteq \A$ since each of $\nu_x, \nu_y, \nu_z$ is contained in $\A$ by Definition \ref{def:A} and Definition \ref{def:nu}. We now show that $W \supseteq \A$. By Definition \ref{def:A}, $\A$ is spanned by the elements (\ref{eq:Abasis}). Let $x^ry^sz^t$ denote an element of (\ref{eq:Abasis}). Then $r+s+t = 2n$ is even. Write $x^r y^s z^t = g_1 g_2 \cdots g_n$ such that $g_i$ is among
\begin{equation}\label{eq:baseterms}
x^2, \quad xy, \quad y^2, \quad yz, \quad z^2, \quad xz
\end{equation}
for $1 \le i \le n$. By Lemma \ref{lem:rel1} and Lemma \ref{lem:x2innu}, each term in (\ref{eq:baseterms}) is contained in $W$. Therefore $W$ contains $g_i$ for $1\le i \le n$, so $W$ contains $x^r y^s z^t$. Consequently, $W \supseteq \A$, so $W = \A$.
\end{proof}

In the case where $q^4 \ne 1$, we have a second generating set of interest.
\begin{corollary}\label{cor:x2genset}
The $\K$-algebra $\A$ is generated by $x^2, y^2, z^2$ provided that $q^4 \ne 1$.
\end{corollary}
\begin{proof}
By Lemma \ref{lem:x2innu}, each of $x^2, y^2, z^2$ is contained in the subalgebra of $U_q(\mathfrak{sl}_2)$ generated by $\nu_x, \nu_y, \nu_z$. By Lemma \ref{lem:nuinx2}, each of $\nu_x, \nu_y, \nu_z$ is contained in the subalgebra of $U_q(\mathfrak{sl}_2)$ generated by $x^2, y^2, z^2$. Thus, $\nu_x, \nu_y, \nu_z$ and $x^2, y^2, z^2$ generate the same subalgebra of $U_q(\mathfrak{sl}_2)$. By Lemma \ref{lem:nugenset}, this implies that $x^2, y^2, z^2$ generate $\A$.
\end{proof}

%

\begin{lemma}\label{lem:Abasis}
The elements {\rm (\ref{eq:Abasis})} are a basis for the $\K$-vector space $\A$.
\end{lemma}
\begin{proof}
Follows from Definition \ref{def:A} and Lemma \ref{lem:Uqsl2basis}.
\end{proof}

It will be useful to have a basis for the $\K$-vector space $\A$ expressed in terms of the elements  $\nu_x, \nu_y, \nu_z$ and $x^2, y^2, z^2$ of $\A$.
\begin{lemma}
The following is a basis for the $\K$-vector space $\A$:
\begin{equation}\label{eq:allowedwords}
\begin{aligned}
    x^{2r} \nu_z^{\delta_1} y^{2s} \nu_x^{\delta_2} z^{2t} \qquad &r, s, t \in \N, \ \delta_1, \delta_2 \in \{0, 1\},\\
    x^{2r} \nu_y z^{2t} \qquad &r, t \in \N.
\end{aligned}
\end{equation}
\end{lemma}
\begin{proof}
For $n \in \N$, let $W_n$ denote the $\K$-subspace of $\A$ spanned by the elements $x^r y^s z^t$ of (\ref{eq:Abasis}) that satisfy $r+s+t = n$. By Lemma \ref{lem:Abasis}, the sum $\A = \sum_{n=0}^\infty W_n$ is direct.

\medskip
Let $w$ be an element of (\ref{eq:allowedwords}). First, assume $w = x^{2r}\nu_z^{\delta_1} y^{2s} \nu_x^{\delta_2} y^{2t}$ and let $N = 2(r+s+t+\delta_1+\delta_2)$. Using Definition \ref{def:nu}, we find that $w \in \sum_{n=0}^N W_n$ and

\[
	w - q^{\delta_1+\delta_2} x^{2r+\delta_1}y^{2s+\delta_1+\delta_2}z^{2t+\delta_2} \in \sum_{n=0}^{N-1}W_n.
\]

Next, assume $w = x^{2r} \nu_y z^{2t}$ and let $N = 2(r+t+1)$. Using Definition \ref{def:nu}, we find that $w \in \sum_{n=0}^N W_n$ and
\[
	w - q^{-1}x^{2r+1}z^{2t+1} \in \sum_{n=0}^{N-1} W_n.
\]

The set
\begin{align*}
x^{2r+\delta_1}y^{2s+\delta_1+\delta_2}z^{2t+\delta_2} \qquad &r, s,t \in \N, \ \delta_1, \delta_2 \in \{0, 1\}\\
x^{2r+1} z^{2t+1} \qquad &r, t \in \N
\end{align*}
is exactly the basis (\ref{eq:Abasis}). The result follows.
\end{proof}

\section{A presentation for \texorpdfstring{$\A$}{A}}

We saw in Lemma \ref{lem:nugenset} that $\nu_x, \nu_y, \nu_z$ generate the $\K$-algebra $\A$. In Lemma \ref{lem:defrel1} and Lemma \ref{lem:defrel2}, we gave twelve relations (\ref{eq:defrel1a})--(\ref{eq:defrel2f}) satisfied by $\nu_x, \nu_y, \nu_z$. In this section, we show that the generators $\nu_x, \nu_y, \nu_z$ and relations (\ref{eq:defrel1a})--(\ref{eq:defrel2f}) give a presentation for the $\K$-algebra $\A$.

\begin{definition}\label{def:Aprime}\rm
Let $\A^\prime$ denote the unital associative $\K$-algebra with generators $\nu_x, \nu_y, \nu_z$ and relations {\rm (\ref{eq:defrel1a})--(\ref{eq:defrel2f})}.
\end{definition}

We will show that the $\K$-algebras $\A$ and $\A^\prime$ are isomorphic. We define elements $x^2, y^2, z^2$ of $\A^\prime$ in the following way.

\begin{definition}\label{def:squaredinAprime} \rm
Define elements $x^2, y^2, z^2$ of $\A^\prime$ by
\begin{align*}
x^2 &= 1 - \frac{q\nu_y\nu_z - q^{-1}\nu_z\nu_y}{q-q^{-1}},\\
y^2 &= 1 - \frac{q\nu_z\nu_x - q^{-1}\nu_x\nu_z}{q-q^{-1}},\\
z^2 &= 1 - \frac{q\nu_x\nu_y - q^{-1}\nu_y\nu_x}{q-q^{-1}}.
\end{align*}
\end{definition}

\begin{note}\rm
Referring to Definition \ref{def:squaredinAprime}, we emphasize that we are viewing $x^2, y^2, z^2$ as symbols representing elements of $\A^\prime$.  We do not attach any meaning to the symbols $x, y, z$ in the context of $\A^\prime$.
\end{note}

\medskip
In Lemmas \ref{lem:rel2} -- \ref{lem:x2innu}, we gave relations satisfied by the elements $\nu_x, \nu_y, \nu_z, x^2, y^2, z^2$ of $U_q(\mathfrak{sl}_2)$. We now show that these relations are satisfied by the corresponding elements of $\A^\prime$.

\begin{lemma}\label{lem:Aprimerels1a}
The relations {\rm (\ref{eq:x2innua})--(\ref{eq:x2innuc})} hold in $\A^\prime$.
\end{lemma}
\begin{proof}
Immediate from Definition \ref{def:squaredinAprime}.
\end{proof}

\begin{lemma}  \label{lem:Aprimerels1}
The relations {\rm (\ref{eq:rel2a})--(\ref{eq:rel2c})} hold in $\A^\prime$.
\end{lemma}
\begin{proof}
To show that the equation on the left in (\ref{eq:rel2a}) holds in $\A^\prime$, we show that $x^2 \nu_y - q^4 \nu_y x^2 = 0$. By Definition \ref{def:squaredinAprime} and (\ref{eq:defrel1b}),
 \begin{align*}
    x^2 \nu_y - q^4 \nu_y x^2 &= \left(1 - \frac{q\nu_y\nu_z - q^{-1}\nu_z\nu_y}{q-q^{-1}}\right)\nu_y - q^4 \nu_y \left(1 - \frac{q\nu_y\nu_z - q^{-1}\nu_z\nu_y}{q-q^{-1}}\right)\\
        &= (1-q^4)\nu_y + q^2 \frac{q^3\nu_y^2\nu_z - (q+q^{-1})\nu_y\nu_z\nu_y + q^{-3}\nu_z\nu_y^2}{q-q^{-1}}\\
        &= (1-q^4)\nu_y + q^2 \frac{(q^2-q^{-2})(q-q^{-1})\nu_y}{q-q^{-1}}\\
        &= 0.
 \end{align*}
The remaining relations are similarly obtained.
\end{proof}

\begin{lemma} \label{lem:Aprimerels2}
The relations {\rm (\ref{eq:x2nuxa})--(\ref{eq:x2nuxf})} hold in $\A^\prime$.
\end{lemma}
\begin{proof}
We show that (\ref{eq:x2nuxa}) holds in $\A^\prime$. By Definition \ref{def:squaredinAprime} and (\ref{eq:defrel2d}),
\begin{align*}
x^2 \nu_x &= \left(1-\frac{q\nu_y\nu_z - q^{-1}\nu_z\nu_y}{q-q^{-1}}\right)\nu_x\\
        &= \nu_x - \left(\nu_x - q^2\nu_y - q^{-2}\nu_z + \frac{q^2\nu_y\nu_z - q^{-2}\nu_z\nu_y}{q-q^{-1}}\right)\\
        &= q^2 \nu_y + q^{-2}\nu_z - q^{-1}\frac{q\nu_y\nu_z-q^{-1}\nu_z\nu_y}{q-q^{-1}}-q\nu_y\nu_z\\
        &= q^2 \nu_y + q^{-2}\nu_z - q^{-1}(1-x^2)-q\nu_y\nu_z\\
        &= q^{-1}x^2 - q^{-1} + q^2 \nu_y + q^{-2}\nu_z - q\nu_y\nu_z.
\end{align*}
The remaining relations are similarly obtained.
\end{proof}

\begin{lemma}  \label{lem:Aprimerels3}
The relations {\rm (\ref{eq:squareda})--(\ref{eq:squaredf})} hold in $\A^\prime$.
\end{lemma}
\begin{proof}
We show that (\ref{eq:squareda}) holds in $\A^\prime$. By Definition \ref{def:squaredinAprime},
\begin{equation}\label{eq:Aprimerels31}
 x^2 y^2 = x^2 \left(1-\frac{q\nu_z\nu_x - q^{-1}\nu_x\nu_z}{q-q^{-1}}\right) = x^2 - \frac{qx^2 \nu_z\nu_x - q^{-1}x^2 \nu_x \nu_z}{q-q^{-1}}.
\end{equation}
By Lemma \ref{lem:Aprimerels1} and Lemma \ref{lem:Aprimerels2}, the relations  (\ref{eq:rel2a}) and (\ref{eq:x2nuxa}) hold in $\A^\prime$. By  (\ref{eq:rel2a}) and (\ref{eq:x2nuxa}),
\begin{align}
qx^2 \nu_z\nu_x &= q^{-4}\nu_z x^2 - q^{-4}\nu_z + q^{-1}\nu_z\nu_y + q^{-5}\nu_z^2 - q^{-2}\nu_z\nu_y\nu_z, \label{eq:Aprimerels32}\\
q^{-1}x^2 \nu_x \nu_z &= q^{-6}\nu_z x^2 - q^{-2}\nu_z + q\nu_y\nu_z + q^{-3}\nu_z^2 - \nu_y \nu_z^2. \label{eq:Aprimerels33}
\end{align}
By (\ref{eq:Aprimerels32}) and (\ref{eq:Aprimerels33}), it follows that
\begin{equation} \label{eq:Aprimerels34}
\frac{qx^2 \nu_z\nu_x - q^{-1}x^2 \nu_x \nu_z}{q-q^{-1}} = q^{-5}\nu_z x^2 + q^{-3}\nu_z - q^{-4}\nu_z^2 - \frac{q\nu_y \nu_z - q^{-1}\nu_z\nu_y}{q-q^{-1}} + \frac{\nu_y\nu_z^2-q^{-2}\nu_z\nu_y\nu_z}{q-q^{-1}}.
\end{equation}
By  (\ref{eq:defrel1e}),
\begin{equation}\label{eq:Aprimerels35}
\frac{\nu_y\nu_z^2-q^{-2}\nu_z\nu_y\nu_z}{q-q^{-1}} = q^{-5}\nu_z \frac{q\nu_y \nu_z-q^{-1}\nu_z\nu_y}{q-q^{-1}} + q^{-3}(q^2-q^{-2})\nu_z.
\end{equation}
By Definition \ref{def:squaredinAprime},
\begin{equation}\label{eq:Aprimerels36}
 \frac{q\nu_y \nu_z-q^{-1}\nu_z\nu_y}{q-q^{-1}} = 1-x^2.
\end{equation}
Simplifying (\ref{eq:Aprimerels34}) using (\ref{eq:Aprimerels35}) and (\ref{eq:Aprimerels36}), we get
\begin{align}
\frac{qx^2 \nu_z\nu_x - q^{-1}x^2 \nu_x \nu_z}{q-q^{-1}} &= q^{-5}\nu_z x^2 + (q^{-1}+q^{3}-q^{-5})\nu_z + q^{-4}\nu_z^2 + (q^{-5}\nu_z)(1-x^2)\notag\\
&= x^2 - (1-q^{-2}(q+q^{-1})\nu_z + q^{-4}\nu_z^2). \label{eq:Aprimerels37}
\end{align}
Combining (\ref{eq:Aprimerels31}) and (\ref{eq:Aprimerels37}), we get
\begin{equation*}
x^2 y^2 = 1-q^{-2}(q+q^{-1})\nu_z + q^{-4}\nu_z^2.
\end{equation*}
The remaining relations are similarly obtained.
\end{proof}

Consider the following elements of $\A^\prime$:
\begin{equation}\label{eq:wordelts}
\nu_x, \quad \nu_y, \quad \nu_z, \quad x^2, \quad y^2, \quad z^2.
\end{equation}

\begin{definition}\rm
For $n \ge 0$, define an $\A^\prime$-word of length $n$ to be a product $g_1g_2\cdots g_n$ in $\A^\prime$ such that $g_i$ is among (\ref{eq:wordelts}) for $1 \le i \le n$. We interpret the $\A^\prime$-word of length zero to be the multiplicative identity in $\A^\prime$.
\end{definition}

We now define two conditions on an $\A^\prime$-word, called the forbidden and allowed conditions. We begin by defining these conditions on an $\A^\prime$-word of length 2.
\begin{definition}\label{def:forbidden1}\rm
For any $\A^\prime$-word $g_1g_2$ of length 2, consider the entry in the following table with row $g_1$ and column $g_2$. We say that $g_1g_2$ is {\it forbidden} whenever the entry has a \xmark \ and {\it allowed} whenever the entry has a \cmark. Observe that an $\A^\prime$-word of length 2 is allowed whenever it is not forbidden.
\begin{center}
\begin{tabular}{c|ccc|ccc}
& $\nu_x$ & $\nu_y$ & $\nu_z$ & $x^2$ & $y^2$ & $z^2$\\
\hline
$\nu_x$ & \xmark & \xmark & \xmark & \xmark & \xmark & \cmark\\
$\nu_y$ & \xmark & \xmark & \xmark & \xmark & \xmark & \cmark\\
$\nu_z$ & \cmark & \xmark & \xmark & \xmark & \cmark & \cmark\\
\hline
$x^2$ & \cmark & \cmark & \cmark & \cmark & \cmark & \cmark\\
$y^2$ & \cmark & \xmark & \xmark & \xmark & \cmark & \cmark\\
$z^2$ & \xmark & \xmark & \xmark & \xmark & \xmark & \cmark
\end{tabular}
\end{center}
\end{definition}

\begin{definition}\label{def:forbidden2}\rm
For $n \ge 0$, we say that an $\A^\prime$-word $w = g_1 g_2 \cdots g_n$ is {\it forbidden} whenever there exists $1 \le i \le n-1$ such that the $\A^\prime$-word $g_i g_{i+1}$ is forbidden. We say that $w$ is {\it allowed} whenever $w$ is not forbidden.
\end{definition}

Referring to Definition \ref{def:forbidden2}, the notion of {\it allowed} has the following interpretation. Define a map from the set of $\A^\prime$-words to $\A$ that sends $w \mapsto \overline{w}$, where
\begin{align*}
	&\overline{x^2} = x^2, \quad \overline{y^2} = y^2, \quad \overline{z^2} = z^2,\\
	&\overline{\nu_x} = yz, \quad \overline{\nu_y} = xz, \quad \overline{\nu_z} = xy,
\end{align*}
and for an $\A^\prime$-word $w = g_1 g_2 \cdots g_n$,
\[
\overline{w} = \overline{g_1}\  \overline{g_2} \cdots \overline{g_n}.
\]
An $\A^\prime$-word $w$ is allowed whenever $\overline{w} = x^r y^s z^t$ for some $r, s, t \in \N$.  For example, $x^2\nu_y$ is allowed because $\overline{x^2 \nu_y} = x^3 z$, but $\nu_x\nu_y$ is forbidden because $\overline{\nu_x\nu_y} = yzxz$.

\begin{lemma}\label{lem:allowed}
An $\A^\prime$-word is allowed if and only if it appears in {\rm (\ref{eq:allowedwords})}.
\end{lemma}
\begin{proof}
Immediate from the above comments.
\end{proof}

\begin{lemma}\label{lem:spanningset}
The $\K$-vector space $\A^\prime$ is spanned by its allowed words.
\end{lemma}
\begin{proof}
The $\K$-algebra $\A^\prime$ is generated by $\nu_x, \nu_y, \nu_z$, so every element of $\A^\prime$ can be written as a linear combination of $\A^\prime$-words. It suffices to show that every $\A^\prime$-word can be expressed as a linear combination of allowed $\A^\prime$-words.

\medskip
We first show that every $\A^\prime$-word of length 2 can be expressed as a linear combination of allowed $\A^\prime$-words. By Lemmas \ref{lem:Aprimerels1a}--\ref{lem:Aprimerels3}, the relations (\ref{eq:rel2a})--(\ref{eq:squaredf}) and  (\ref{eq:x2nuxa})--(\ref{eq:x2innuc}) hold in $\A^\prime$.
Each forbidden $\A^\prime$-word of length 2 can be expressed as a linear combination of allowed $\A^\prime$-words by using the relations listed in the following table:

\medskip
\begin{center}
\begin{tabular}{c|ccc|ccc}
& $\nu_x$ & $\nu_y$ & $\nu_z$ & $x^2$ & $y^2$ & $z^2$\\
\hline
$\nu_x$ & (\ref{eq:squaredb}) & (\ref{eq:x2nuxf}) & (\ref{eq:x2innub}) & (\ref{eq:x2nuxd}), (\ref{eq:x2nuxa}) & (\ref{eq:rel2b}) & \cmark\\
$\nu_y$ & (\ref{eq:x2innuc}), (\ref{eq:x2nuxf}) & (\ref{eq:squaredf}) & (\ref{eq:x2nuxa}) & (\ref{eq:rel2a}) & (\ref{eq:x2nuxe}) & \cmark\\
$\nu_z$ & \cmark & (\ref{eq:x2innua}), (\ref{eq:x2nuxa}) & (\ref{eq:squareda}) & (\ref{eq:rel2a}) & \cmark & \cmark\\
\hline
$x^2$ & \cmark & \cmark & \cmark & \cmark & \cmark & \cmark\\
$y^2$ & \cmark & (\ref{eq:x2nuxb}) & (\ref{eq:rel2b}) & (\ref{eq:squaredd}), (\ref{eq:squareda}) & \cmark & \cmark\\
$z^2$ & (\ref{eq:rel2c}) & (\ref{eq:rel2c}) & (\ref{eq:x2nuxc}), (\ref{eq:x2nuxf}) & (\ref{eq:squaredc}), (\ref{eq:squaredf}) & (\ref{eq:squarede}), (\ref{eq:squaredb}) & \cmark
\end{tabular}
\end{center}

\medskip
For example, solving for $\nu_x^2$ in (\ref{eq:squaredb}) gives an expression for $\nu_x^2$ as a linear combination of allowed $\A^\prime$-words.  In cases where there are two relations listed, the expression resulting from the first relation contains a forbidden $\A^\prime$-word of length 2 that can be eliminated using the second relation.

\medskip
For each entry in the table above, we obtain an equation with a forbidden $\A^\prime$-word of length 2 on one side and an equivalent linear combination of allowed $\A^\prime$-words on the other side. We call these equations the {\it reduction rules} for $\A^\prime$. Specifically, for a forbidden $\A^\prime$-word $g_1 g_2$, the reduction rule for $g_1 g_2$ is a linear combination involving one allowed $\A^\prime$-word of length 2, which we denote $\tilde{g}_1 \tilde{g}_2$, and $\A^\prime$-words of length 0 and 1.
The 21 reduction rules can be found in Section \ref{sec:appendix}.

\medskip
Now let $w = g_1 g_2 \cdots g_n$ be an $\A^\prime$-word of length $n > 2$. By a {\it forbidden pair} for $w$, we mean an ordered pair of integers $(i,j)$ such that $1 \le i < j \le n$ and the word $g_i g_j$ is forbidden. For $n \ge 0$, let $W_n$ denote the set of $\A^\prime$-words of length $n$ and let $W = \bigcup_{n=0}^{\infty} W_n$ denote the set of all $\A^\prime$-words. We now define a partial order $<$ on $W$. The definition has two aspects. First, for $n > m \ge 0$, every word in $W_m$ is less than every word in $W_n$ with respect to $<$. Second, for $n > 0$ and for $w, w^\prime \in W_n$, $w < w^\prime$ whenever $w$ has fewer forbidden pairs than $w^\prime$.

\medskip
Let $w = g_1 g_2 \cdots g_n$ denote a forbidden $\A^\prime$-word. Then there exists an integer $1 \le i \le n-1$ such that $g_i g_{i+1}$ is forbidden.  Using the reduction rule for $g_i g_{i+1}$, we can replace $g_i g_{i+1}$ in $w$ with a linear combination of allowed $\A^\prime$-words. Let $w^\prime$ be a word appearing in the resulting linear combination. Then $w^\prime$ has length $n, n-1$, or $n-2$. First assume that $w^\prime$ has length $n-1$ or $n-2$. By definition of the partial order, $w^\prime < w$.

\medskip
Next, assume $w^\prime$ has length $n$. Then $w^\prime = g_1 \cdots g_{i-1} \tilde{g}_i \tilde{g}_{i+1} g_{i+2} \cdots g_{n}$. We show that $w^\prime$ has fewer forbidden pairs than $w$. For 12 of the 21 reduction rules, $\tilde{g}_1 \tilde{g}_2 = g_2 g_1$. In these cases, the forbidden pairs of $w^\prime$ are exactly the forbidden pairs of $w$ other than $(i, i+1)$. Thus $w^\prime$ has one fewer forbidden pair than $w$, so $w^\prime < w$.

\medskip
In the remaining 9 cases, let $1 \le j \le n$ with $j \not\in \{i, i+1\}$. It is routinely verified that if neither of $(j, i), (j,i+1)$ (resp. $(i, j), (i+1, j)$) is a forbidden pair in $w$, then neither is a forbidden pair in $w^\prime$. Similarly, it is routinely verified that if exactly one of $(j, i), (j, i+1)$ (resp. $(i, j), (i+1, j)$) is a forbidden pair in $w$, then $(j,i)$ (resp. $(i+1, j)$) is a forbidden pair in $w^\prime$ and $(j, i+1)$ (resp. $(i, j)$) is not a forbidden pair in $w^\prime$. Thus, the number of forbidden pairs of the form $(j, i), (j, i+1), (i, j), (i+1, j)$ in $w^\prime$ is less than or equal to the number of such forbidden pairs in $w$. Observe that the forbidden pairs of $w^\prime$ not involving $i, i+1$ are exactly the forbidden pairs of $w$ not involving $i, i+1$. Also, note that $(i, i+1)$ is a forbidden pair in $w$ and not a forbidden pair in $w^\prime$. Thus $w^\prime$ has strictly fewer forbidden pairs than $w$, so $w^\prime < w$.

\medskip
Therefore, each word in the linear combination obtained by applying the reduction rule is strictly less than $w$ with respect to the partial order $<$. As a result, we can iteratively replace forbidden $\A^\prime$-subwords of length 2 using the reduction rules and the process will terminate after a finite number of steps. The result of this process is an expression for $w$ as a linear combination of allowed $\A^\prime$-words.
\end{proof}

\begin{theorem}
The $\K$-algebra $\A^\prime$ from Definition \ref{def:Aprime} and the $\K$-algebra $\A$ from Definition \ref{def:A} are isomorphic. An isomorphism is given by $\nu_\eta \mapsto \nu_\eta$ for $\eta \in \{x, y, z\}$.
\end{theorem}
\begin{proof}
By Lemma \ref{lem:defrel1} and Lemma \ref{lem:defrel2}, the elements $\nu_x, \nu_y, \nu_z$ of $\A$ satisfy the defining relations (\ref{eq:defrel1a})--(\ref{eq:defrel2f}) for $\A^\prime$. Therefore there exists a $\K$-algebra homomorphism $\phi : \A^\prime \to \A$ that sends $\nu_\eta \mapsto \nu_\eta$ for $\eta \in \{x, y, z\}$. To show that $\phi$ is an isomorphism, we show that $\phi$ maps a basis for $\A^\prime$ to a basis for $\A$.

\medskip
Let $W$ denote the set of allowed $\A^\prime$-words. By Lemma \ref{lem:allowed}, $\phi$ sends the elements of $W$ to the elements of the basis (\ref{eq:allowedwords}) for $\A$. Since the elements of (\ref{eq:allowedwords}) are linearly independent in $\A$, the elements of $W$ are linearly independent in $\A^\prime$. Together with Lemma \ref{lem:spanningset}, this shows that $W$ is a basis for $\A^\prime$. Consequently, $\phi$ sends a basis for $\A^\prime$ to a basis for $\A$, so $\phi$ is an isomorphism.
\end{proof}

\section{Finite-dimensional irreducible \texorpdfstring{$U_q(\mathfrak{sl}_2)$}{Uq(sl2)}-modules}

Our next main goal is to classify up to isomorphism the finite-dimensional irreducible $\A$-modules.  We begin by recalling the finite-dimensional irreducible $U_q(\mathfrak{sl}_2)$-modules. For the rest of the paper, we assume that $q$ is not a root of unity.

\begin{lemma}{\rm \cite[Lemma 4.1, 4.2]{Ito2006}}\label{lem:uqsl2modules} 
There exists a family of finite-dimensional irreducible $U_q(\mathfrak{sl}_2)$-modules
\begin{equation}\label{eq:Lde}
    L(d, \epsilon), \qquad \epsilon \in \{1, -1\}, \ d \in \N,
\end{equation}
with the following property: $L(d, \epsilon)$ has a basis $\{u_i\}_{i=0}^d$ such that for $0 \le i \le d$,
\begin{align*}
(\epsilon x - q^{d-2i})u_{i} &= (q^{-d}-q^{d-2i+2})u_{i-1},\\
(\epsilon y - q^{d-2i})u_{i} &= (q^d-q^{d-2i-2})u_{i+1}, \\
\epsilon z u_{i} &= q^{2i-d}u_{i},
\end{align*}
where $u_{-1} = 0, u_{d+1} = 0$. For {\rm char} $\K = 2$, we interpret the set $\{1, -1\}$ as $\{1\}$. Every finite-dimensional irreducible $U_q(\mathfrak{sl}_2)$-module is isomorphic to exactly one of the modules {\rm (\ref{eq:Lde})}.
\end{lemma}

\medskip
We now consider the action of $\A$ on the $U_q(\mathfrak{sl}_2)$-module $L(d, \epsilon)$. Since $q$ is not a root of unity, we have two generating sets of interest for $\A$: the generating set $\nu_x, \nu_y, \nu_z$ from Lemma \ref{lem:nugenset} and the generating set $x^2, y^2, z^2$ from Corollary \ref{cor:x2genset}. We focus on the actions of these $\A$-generators.
\begin{lemma}\label{lem:uqsl2actions1}
For $d \in \N$ and $\epsilon \in \{1, -1\}$, the $\A$-generators $\nu_x, \nu_y, \nu_z$ act on $L(d, \epsilon)$ in the following way. For $0 \le i \le d$,
\begin{enumerate}
\item[{\rm (i)}] $\nu_x u_i = q^{-1}(1 - q^{2(i+1)})u_{i+1}$,
\item[{\rm (ii)}] $\nu_y u_i = q(1-q^{2(i-d-1)})u_{i-1}$,
\item[{\rm (iii)}] $\nu_z u_i$ is a linear combination of $u_{i-1}, u_i, u_{i+1}$ with the following coefficients:
\begin{center}
\begin{tabular}{c|l}
{\rm term} & {\rm coefficient}\\
\hline
$u_{i-1}$ & $q^{2d-4i+3}(1-q^{2(i-d-1)})$\\
$u_{i}$ & $(q^{2d-2i+1}+q^{-2i-1}-q^{2d-4i+1}-q^{2d-4i-1})$\\
$u_{i+1}$ & $q^{2d-4i-3}(1-q^{2(i+1)})$
\end{tabular}
\end{center}	
\end{enumerate}
\end{lemma}
\begin{proof}
Follows from Definition \ref{def:nu} and Lemma \ref{lem:uqsl2modules}.
\end{proof}

\begin{lemma}\label{lem:uqsl2actions2}
For $d \in \N$ and $\epsilon \in \{1, -1\}$, the $\A$-generators $x^2, y^2, z^2$ act on $L(d, \epsilon)$ in the following way. For $0 \le i \le d$,
\begin{enumerate}
\item[{\rm (i)}] $x^2u_i$ is a linear combination of $u_{i-2}, u_{i-1}, u_{i}$ with the following coefficients:
\begin{center}
\begin{tabular}{c|l}
{\rm term} & {\rm coefficient}\\
\hline
$u_{i-2}$ & $(q^{-d}-q^{d-2i+2})(q^{-d}-q^{d-2i+4})$\\
$u_{i-1}$ & $q^{d-2i+1}(q+q^{-1})(q^{-d}-d^{d-2i+2})$\\
$u_{i}$ & $q^{2d-4i}$
\end{tabular}
\end{center}

\item[{\rm (ii)}] $y^2u_i$ is a linear combination of $u_{i}, u_{i+1}, u_{i+2}$ with the following coefficients:
\begin{center}
\begin{tabular}{c|l}
{\rm term} & {\rm coefficient}\\
\hline
$u_i$ & $q^{2d-4i}$\\
$u_{i+1}$ & $q^{d-2i-1}(q+q^{-1})(q^d - q^{d-2i-2})$\\
$u_{i+2}$ & $(q^d-q^{d-2i-2})(q^d-q^{d-2i-4})$
\end{tabular}
\end{center}

\item[{\rm (iii)}] $z^2u_i = q^{4i-2d} u_i$.
\end{enumerate}

\end{lemma}
\begin{proof}
Follows from Lemma \ref{lem:uqsl2modules}.
\end{proof}

\section{From \texorpdfstring{$U_q(\mathfrak{sl}_2)$}{Uq(sl2)}-modules to \texorpdfstring{$\A$}{A}-modules}

Let $V$ denote a $U_q(\mathfrak{sl}_2)$-module. By restricting from $U_q(\mathfrak{sl}_2)$ to $\A$, the $U_q(\mathfrak{sl}_2)$-module $V$ becomes an $\A$-module. We say that the $U_q(\mathfrak{sl}_2)$-module $V$ {\it extends} the $\A$-module $V$. Recall the $U_q(\mathfrak{sl}_2)$-module $L(d, \epsilon)$ from Lemma \ref{lem:uqsl2modules}. In this section, we discuss the $\A$-module $L(d, \epsilon)$.

\begin{lemma}\label{lem:isoAmods}
For $d \in \N$, there exists an $\A$-module isomorphism $L(d, 1) \to L(d, -1)$ that sends $u_i \mapsto u_i$ for $0 \le i \le d$.
\end{lemma}
\begin{proof}
By Lemma \ref{lem:uqsl2actions1}, the actions of $\nu_x, \nu_y, \nu_z$ on $L(d, \epsilon)$ are independent of $\epsilon$. By Lemma \ref{lem:nugenset}, $\nu_x, \nu_y, \nu_z$ generate $\A$, so the action of $\A$ on $L(d, \epsilon)$ is independent of $\epsilon$. The result follows.
\end{proof}

\begin{definition}\label{def:Ld}\rm
For $d \in \N$, we identify $\A$-modules $L(d,1)$, $L(d,-1)$ via the isomorphism in Lemma \ref{lem:isoAmods}. We call the resulting $\A$-module $L(d)$.
\end{definition}

Observe that $L(d)$ has a basis $\{u_i\}_{i=0}^d$ on which $\nu_x, \nu_y, \nu_z$ act as in Lemma \ref{lem:uqsl2actions1} and $x^2, y^2, z^2$ act as in Lemma \ref{lem:uqsl2actions2}.

\begin{lemma}
For $d \in \N$, the $\A$-module $L(d)$ is irreducible.
\end{lemma}
\begin{proof}
Let $W$ denote an nonzero $\A$-submodule of $L(d)$. We show that $W = L(d)$.
For $0 \le i \le d$, let $V_i$ denote the eigenspace of the $z^2$-action on $L(d)$ with eigenvalue $q^{4i-2d}$. By Lemma \ref{lem:uqsl2actions2}, $V = \sum_{i=0}^d V_i$ (direct sum) and $V_i = \text{span}\{u_i\}$ for $0 \le i \le d$. Let $W_i$ denote the projection of $W$ onto $V_i$. Since $W \ne 0$, there exists $k$ such that $W_k \ne 0$. Then $u_k \in W_k$, so $u_k \in W$. Since $W$ is an $\A$-submodule, $\nu_x^n u_k \in W$ and $\nu_y^n u_k \in W$ for all $n \in \N$. By Lemma \ref{lem:uqsl2actions1}, it follows that $u_i \in W$ for all $0 \le i \le d$. Therefore $W = L(d)$.
\end{proof}

\section{Finite-dimensional irreducible \texorpdfstring{$\A$}{A}-modules}

In this section, we classify up to isomorphism the finite-dimensional irreducible $\A$-modules.

\begin{lemma}\label{lem:nilpotent}
Let $V$ be a finite-dimensional $\A$-module. Then the actions of $\nu_x, \nu_y, \nu_z$ on $V$ are nilpotent.
\end{lemma}
\begin{proof}
We show that the $\nu_y$-action on $V$ is nilpotent. Suppose that the $\nu_y$-action on $V$ is not nilpotent.  Then there exists $0 \ne \lambda \in \K$ such that $\lambda$ is an eigenvalue for $\nu_y$.  Since $V$ is finite-dimensional, there exist $M, N \ge 0$ maximal such that $q^{-4M}\lambda$ and $q^{4N}\lambda$ are eigenvalues for $\nu_y$.

\medskip
Let $v \in V$ be a nonzero eigenvector for $\nu_y$ corresponding to eigenvalue $q^{4N}\lambda$.  By Lemma \ref{lem:rel1},
\[
\nu_y z^2 v = q^4 z^2 \nu_y v = q^{4(N+1)}\lambda z^2 v.
\]
By the maximality of $N$, $q^{4(N+1)}\lambda$ is not an eigenvalue for $\nu_y$, so $z^2 v = 0$.  By Lemma \ref{lem:rel2},
\begin{align*}
0 = x^2 z^2 v &= (1-q^2(q+q^{-1})\nu_y + q^4 \nu_y^2) v\\
            &= (1-q^{4N+2}(q+q^{-1})\lambda + q^{8N+4}\lambda^2)v\\
            &= (1-q^{4N+3}\lambda)(1-q^{4N+1}\lambda) v.
\end{align*}
Therefore $\lambda \in \{q^{-(4N+3)},q^{-(4N+1)}\}$.

\medskip
Now let $w \in V$ be a nonzero eigenvector for $\nu_y$ corresponding to eigenvalue $q^{-4M}\lambda$.  By Lemma \ref{lem:rel1},
\[
\nu_y x^2 w = q^{-4} x^2 \nu_y w = q^{-4(M+1)}\lambda x^2 w.
\]
By the maximality of $M$, $q^{-4(M+1)}\lambda$ is not an eigenvalue for $\nu_y$, so $x^2 w = 0$.  By Lemma \ref{lem:rel2},
\begin{align*}
0 = z^2x^2  w &= (1-q^{-2}(q+q^{-1})\nu_y + q^{-4}\nu_y^2) w\\
            &= (1-q^{-4M-2}(q+q^{-1})\lambda + q^{-8M-4}\lambda^2) w\\
            &= (1 - q^{-4M-3}\lambda)(1-q^{-4M-1}\lambda) w.
\end{align*}
Therefore $\lambda \in \{q^{4M+3}, q^{4M+1}\}$.  However, since $-(4N+3), -(4N+1)$ are negative,  $4M+3, 4M+1$ are positive, and $q$ is not a root of unity, it cannot be the case that $\lambda \in \{q^{-(4N+3)},q^{-(4N+1)}\}$ and $\lambda \in \{q^{4M+3}, q^{4M+1}\}$.  This is a contradiction.  Therefore the action of $\nu_y$ on $V$ is nilpotent.

\medskip
A similar argument shows that the actions of $\nu_x$ and $\nu_z$ on $V$ are nilpotent.
\end{proof}

\begin{lemma}\label{lem:commonevec}
Let $V$ denote a finite-dimensional $\A$-module.  Then the kernel of the $\nu_y$-action on $V$ contains a common eigenvector for $z^2$ and $\nu_y\nu_x$.
\end{lemma}
\begin{proof}
Let $W$ denote the kernel of the $\nu_y$-action on $V$. By Lemma \ref{lem:nilpotent},   $\nu_y$ acts nilpotently on $V$, so $W \ne 0$.  It suffices to show that $W$ is fixed by both $z^2$ and $\nu_y\nu_x$ and that $z^2$ and $\nu_y\nu_x$ commute on $W$.

\medskip
Let $w \in W$.  By Lemma \ref{lem:rel2},
\[
	0 = z^2 \nu_y w = q^{-4}\nu_y z^2 w.
\]
Therefore $\nu_y z^2 w = 0$, so $z^2 w \in W$.

\medskip
By Lemma \ref{lem:defrel1},
\begin{equation}\label{eq:tosimplify}
q^{-3}\nu_y^2\nu_x w- (q+q^{-1})\nu_y\nu_x\nu_y w+ q^3\nu_x\nu_y^2 w= (q^2 - q^{-2})(q-q^{-1})\nu_y w.
\end{equation}
Since $\nu_y w = 0$, the equation (\ref{eq:tosimplify}) simplifies to $q^{-3}\nu_y^2\nu_x w= 0$.  Then  $\nu_y (\nu_y \nu_x w ) = 0$, so $\nu_y \nu_x w \in W$.
Therefore $W$ is fixed by both $z^2$ and $\nu_y\nu_x$.

\medskip
By Lemma \ref{lem:rel2},
\[
	z^2 \nu_y \nu_x = q^{-4} \nu_y z^2 \nu_x = \nu_y \nu_x z^2.
\]
Therefore $z^2$ and $\nu_y\nu_x$ commute in $\A$, so they commute on $W$.  The result follows.
\end{proof}

For the rest of this section, the following notation will be in effect.
\begin{notation}\label{not:setup} \rm
Let $V$ denote a finite-dimensional irreducible $\A$-module. By Lemma \ref{lem:commonevec}, there exists  $0 \ne v_0 \in V$ such that $\nu_y v_0 = 0$ and $v_0$ is a common eigenvector for $z^2$ and $\nu_y\nu_x$.  Define $v_i = \nu_x^i v_0$ for $i \in \N$.  Define $v_{-1} = 0$. Since $\nu_x$ is nilpotent, only finitely many $v_i$ are nonzero. Let $d \in \N$ be maximal such that $v_{d} \ne 0$.
\end{notation}

We will show that $\{v_i\}_{i=0}^d$ is a basis for $V$. To do this, we consider the actions of certain elements of $\A$ on $\{v_i\}_{i=0}^d$.

\begin{lemma}\label{lem:eigenvectors}
With reference to Notation \ref{not:setup}, there exists $0 \ne \lambda \in \K$ such that $z^2 v_i = q^{4i}\lambda v_i$ for $0 \le i \le d$.
\end{lemma}
\begin{proof}
Since $v_0$ is an eigenvector for $z^2$, there exists $\lambda \in \K$ such that $z^2 v_0 = \lambda v_0$. By Lemma \ref{lem:rel2}, for all $0 \le i \le d$,
\[
	z^2 v_i = z^2 \nu_x^{i} v_0 = q^{4i} \nu_x^{i} z^2 v_0 = q^{4i}\lambda \nu_x^i v_0 = q^{4i}\lambda v_i.
\]
Suppose that $\lambda = 0$.  Then by Lemma \ref{lem:rel3},
\[
	0 = x^2 z^2 v_0 = \left(1 - q^2(q+q^{-1})\nu_y + q^4 \nu_y^2\right)v_0 = v_0.
\]
Since $v_0 \ne 0$, this is a contradiction. Therefore $\lambda \ne 0$.
\end{proof}

\begin{lemma}\label{lem:nuyaction}
With reference to Notation \ref{not:setup}, $\nu_y v_i \in {\rm span}\{v_{i-1}\}$ for $0 \le i \le d$.
\end{lemma}
\begin{proof}
We proceed by induction on $i$.  Since $\nu_y v_0 = 0$, $\nu_y v_0 \in {\rm span}\{v_{-1}\}$. Since $v_0$ is an eigenvector of $\nu_y\nu_x$, $\nu_y v_1 = \nu_y \nu_x v_0 \in {\rm span}\{v_0\}$.

\medskip
Now let $2 \le m \le d$ and assume that $\nu_y v_i \in {\rm span}\{v_{i-1}\}$ for $0 \le i < m$.  We show that $\nu_y v_m \in {\rm span}\{v_{m-1}\}$.  Multiplying by $\nu_x^{m-2}$ on the right on both sides of (\ref{eq:defrel1a}), we have
\begin{equation}\label{eq:dimproof1}
\nu_y \nu_x^m = q^3(q^2-q^{-2})(q-q^{-1})\nu_x^{m-1} - q^6 \nu_x^{2}\nu_y\nu_x^{m-2} + q^3(q+q^{-1})\nu_x \nu_y \nu_x^{m-1}.
\end{equation}
Applying both sides of (\ref{eq:dimproof1}) to $v_0$, we have
\begin{equation}\label{eq:nuyvm}
\nu_y v_m = q^3(q^2-q^{-2})(q-q^{-1})v_{m-1} - q^6 \nu_x^{2}\nu_y v_{m-2}+q^3(q+q^{-1})\nu_x \nu_y v_{m-1}.
\end{equation}
By the induction hypothesis, $\nu_y v_{m-2}$ is a scalar multiple of $v_{m-3}$, so $\nu_x^{2}\nu_y v_{m-2}$ is a scalar multiple of $v_{m-1}$.  Similarly, $\nu_y v_{m-1}$ is a scalar multiple of $v_{m-2}$, so $\nu_x \nu_y v_{m-1}$ is a scalar multiple of $v_{m-1}$.  Therefore, the right-hand side of (\ref{eq:nuyvm}) is a scalar multiple of $v_{m-1}$, so $\nu_y v_m \in {\rm span}\{v_{m-1}\}$.
\end{proof}

\begin{lemma}
With reference to Notation \ref{not:setup}, $\{v_i\}_{i=0}^d$ is a basis for $V$.
\end{lemma}
\begin{proof}
By Lemma \ref{lem:eigenvectors}, the elements of $\{v_i\}_{i=0}^d$ are eigenvectors for $z^2$ corresponding to distinct eigenvalues, so they are linearly independent. Thus, it remains to show that $\{v_i\}_{i=0}^d$ span $V$.

\medskip
Let $V^\prime = {\rm span}\{v_i\}_{i=0}^d$. We show that $V^\prime = V$. By Lemma \ref{lem:nugenset}, $\nu_x, \nu_y, \nu_z$ generate $\A$. Thus, it suffices to show that $V^\prime$ is fixed by $\nu_x, \nu_y, \nu_z$.

\medskip
By Notation \ref{not:setup}, $V^\prime$ is fixed by $\nu_x$. By Lemma \ref{lem:nuyaction}, $V^\prime$ is fixed by $\nu_y$.
To show that $V^\prime$ is fixed by $\nu_z$, we consider the $z^2$-action on $V^\prime$.  By Lemma \ref{lem:eigenvectors}, there exists $\lambda \ne 0$ such that $z^2 v_i = q^{4i}\lambda v_i$ for $0 \le i \le d$. Then by Lemma \ref{lem:rel4},
\begin{equation}\label{eq:y2vi}
	\nu_z v_i = q^{-4i}\lambda^{-1} \nu_z z^2 v_i = q^{-4i}\lambda^{-1}\left(q^{-1}z^2-q^{-1}+q^{-2}\nu_x+q^2\nu_y-q\nu_x\nu_y\right)v_i.
\end{equation}
The right-hand side of (\ref{eq:y2vi}) is contained in $V^\prime$ since $V^\prime$ is fixed by $\nu_x, \nu_y$, and $z^2$.  Therefore $V^\prime$ is fixed by $\nu_z$. The result follows.
\end{proof}

We now compute the actions of the elements $\nu_x, \nu_y, \nu_z, x^2, y^2, z^2$ of $\A$ on $\{v_i\}_{i=0}^d$. It is convenient for us to start with $\nu_x, \nu_y, z^2$.

\begin{theorem}\label{thm:actions}
With reference to Notation \ref{not:setup}, the elements $\nu_x, \nu_y, z^2$ of $\A$ act on the basis $\{v_i\}_{i=0}^d$ for $V$ in the following way. For $0 \le i \le d$,
\begin{align}
\nu_x v_i &= v_{i+1} \label{eq:actionsnux},\\
\nu_y v_i &= (q^{2i}-1)(q^{2(i-d-1)}-1)v_{i-1} \label{eq:actionsnuy},\\
z^2 v_i &= q^{4i-2d}v_i \label{eq:actionsz2},
\end{align}
where $v_{-1} = v_{d+1} = 0$.
\end{theorem}
\begin{proof}
Observe that (\ref{eq:actionsnux}) holds by Notation \ref{not:setup}.  By Lemma \ref{lem:eigenvectors}, there exists $0 \ne \lambda \in \K$ such that $z^2 v_i = q^{4i}\lambda v_i$ for $0 \le i \le d$.  By Lemma \ref{lem:nuyaction}, there exist $\alpha_i \in \K$ such that $\nu_y v_i = \alpha_i v_{i-1}$ for $0 \le i \le d$. Since $v_{-1} = v_{d+1} = 0$, we set $\alpha_0 = \alpha_{d+1} = 0$. To verify (\ref{eq:actionsnuy}) and $(\ref{eq:actionsz2})$, it suffices to show that $\lambda = q^{-2d}$ and $\alpha_i = (q^{2i}-1)(q^{2(i-d-1)}-1)$ for $0 \le i \le d$.

\medskip
By Lemma \ref{lem:x2innu},
\begin{equation*}
	q^{4i} \lambda v_i = z^2 v_i = \left(1 - \frac{q \nu_x \nu_y - q^{-1}\nu_y\nu_x}{q-q^{-1}}\right)v_i = \left(1 - \frac{q\alpha_i - q^{-1} \alpha_{i+1}}{q-q^{-1}}\right)v_i .
\end{equation*}

This yields the following recurrence:
\begin{equation}\label{eq:recurrence}
\alpha_{i+1} = q^2 \alpha_i + (q^2 - 1)(q^{4i}\lambda -1) \qquad (0 \le i \le d).
\end{equation}

\medskip
It is easily verified that the solution to the recurrence (\ref{eq:recurrence}) with initial condition $\alpha_0 = 0$ is
\begin{equation}\label{eq:recsolution}
	\alpha_i = 1 - q^{2i} + \lambda (q^{4i-2} - q^{2i-2}) \qquad (0 \le i \le d).
\end{equation}

Setting $i = d+1$ in (\ref{eq:recsolution}) and factoring, we get
\begin{align*}
	\alpha_{d+1} &= (1 - q^{2(d+1)})(1 - \lambda q^{2d}).
\end{align*}
Since $\alpha_{d+1} = 0$, this gives $\lambda = q^{-2d}$. Plugging $\lambda = q^{-2d}$  into (\ref{eq:recsolution}), we get
\begin{align*}
	\alpha_i &= 1 - q^{2i} + q^{-2d}(q^{4i-2} - q^{2i-2})\\
		&= (q^{2i}-1)(q^{2(i-d-1)}-1).
\end{align*}

The result follows.
\end{proof}

\begin{lemma}\label{lem:otheractions}
With reference to Notation \ref{not:setup}, the elements $x^2, y^2, \nu_z$ of $\A$ act on the basis $\{v_i\}_{i=0}^d$ in the following way. For $0 \le i \le d$,
\begin{enumerate}
\item[{\rm (i)}] $x^2v_i$ is a linear combination of $v_{i-2}, v_{i-1}, v_{i}$ with the following coefficients:
\begin{center}
\begin{tabular}{c|l}
{\rm term} & {\rm coefficient}\\
\hline
$v_{i-2}$ & $q^4(q^{2i}-1)(q^{2(i-d-1)}-1)(q^{2(i-1)}-1)(q^{2(i-d-2)}-1)$\\
$v_{i-1}$ & $-q^2(q+q^{-1})(q^{2i}-1)(q^{2(i-d-1)}-1)$\\
$v_{i}$ & $q^{2d-4i}$
\end{tabular}
\end{center}

\item[{\rm (ii)}] $y^2v_i$ is a linear combination of $v_{i}, v_{i+1}, v_{i+2}$ with the following coefficients:
\begin{center}
\begin{tabular}{c|l}
{\rm term} & {\rm coefficient}\\
\hline
$v_i$ & $q^{2d-4i}$\\
$v_{i+1}$ & $- q^{2d-4i-2}(q+q^{-1})$\\
$v_{i+2}$ & $q^{2d-4i-4}$
\end{tabular}
\end{center}

\item[{\rm (iii)}] $\nu_z v_i$ is a linear combination of $v_{i-1}, v_{i}, v_{i+1}$ with the following coefficients:
\begin{center}
\begin{tabular}{c|l}
{\rm term} & {\rm coefficient}\\
\hline
$v_{i-1}$ & $(q^{-2i}-1)(q^{-2(i-d-1)}-1)$\\
$v_{i}$ & $q^{2d-2i+1}+q^{-2i-1}-q^{2d-4i+1}-q^{2d-4i-1}$\\
$v_{i+1}$ & $q^{2d-4i-2}$
\end{tabular}
\end{center}
\end{enumerate}
\end{lemma}
\begin{proof}
By Theorem \ref{thm:actions}, we have $z^2 v_i = q^{4i-2d}v_i$ for $0 \le i \le d$. Thus, $\phi v_i = q^{2d-4i} \phi z^2 v_i$ for each $\phi \in \{x^2, y^2, \nu_z\}$.
Then by Lemma \ref{lem:rel3} and Lemma \ref{lem:rel4},
\begin{align*}
x^2 v_i &= q^{2d-4i}x^2z^2v_i = q^{2d-4i}\left(1-q^2(q+q^{-1})\nu_y+q^4\nu_y^2\right)v_i,\\
y^2 v_i &= q^{2d-4i} y^2 z^2 v_i = q^{2d-4i}\left(1 - q^{-2}(q+q^{-1})\nu_x + q^{-4}\nu_x^2\right)v_i,\\
\nu_zv_i &= q^{2d-4i} \nu_z z^2 v_i = q^{2d-4i}\left(q^{-1}z^2 - q^{-1}+q^{-2}\nu_x + q^2 \nu_y - q \nu_x \nu_y\right)v_i.
\end{align*}

The result follows by simplifying these equations using Theorem \ref{thm:actions}.
\end{proof}

\begin{theorem}
The $\A$-module $V$ from Notation \ref{not:setup} is isomorphic to the $\A$-module $L(d)$ from Definition \ref{def:Ld}. An isomorphism is given by $v_i \mapsto \gamma_i u_i$, where $\gamma_0 = 1$ and $\gamma_{i+1}/\gamma_i = q^{-1}(1 - q^{2(i+1)})$ for $0 \le i \le d-1$.
\end{theorem}
\begin{proof}
Let $\phi : V \to L(d)$ be the map defined by $v_i \mapsto \gamma_i u_i$. Since $q$ is not a root of unity, $\phi$ is a vector space isomorphism. To show that $\phi$ is an $\A$-module isomorphism, it suffices to show that $\phi \nu_\eta = \nu_\eta \phi$ for all $\eta \in \{x, y, z\}$.

\medskip
By Lemma \ref{lem:uqsl2actions1} and Theorem \ref{thm:actions},
\begin{align*}
    \phi \nu_x v_i &= \phi v_{i+1} = \gamma_{i+1}u_{i+1} = \gamma_{i} q^{-1}(1 - q^{2(i+1)}) u_{i+1},\\
    \nu_x \phi v_i &= \gamma_i \nu_x u_i =  \gamma_i q^{-1}(1 - q^{2(i+1)}) u_{i+1}.
\end{align*}
Thus $\phi \nu_x = \nu_x \phi$.

\medskip
By Lemma \ref{lem:uqsl2actions1} and Theorem \ref{thm:actions},
\begin{align*}
\phi \nu_y v_i &= (q^{2i}-1)(q^{2(i-d-1)}-1)\phi v_{i-1} = \gamma_{i-1} (q^{2i}-1)(q^{2(i-d-1)}-1) u_{i-1},\\
\nu_y \phi v_i &= \gamma_i \nu_y v_i = \gamma_i q(1-q^{2i-d-1}) u_{i-1} = \gamma_{i-1} (q^{2i}-1)(q^{2(i-d-1)}-1) u_{i-1}.
\end{align*}
Thus $\phi \nu_y = \nu_y \phi$.

\medskip
By Lemma \ref{lem:uqsl2actions1} and Lemma \ref{lem:otheractions}, $\phi \nu_z v_i$ is a linear combination of $u_{i-1}, u_{i}, u_{i+1}$ with the following coefficients:
\begin{center}
\begin{tabular}{c|l}
{\rm term} & {\rm coefficient}\\
\hline
$u_{i-1}$ & $\gamma_{i-1}(q^{-2i}-1)(q^{-2(i-d-1)}-1)$\\
$u_{i}$ & $\gamma_{i}(q^{2d-2i+1}+q^{-2i-1}-q^{2d-4i+1}-q^{2d-4i-1})$\\
$u_{i+1}$ & $\gamma_{i+1}q^{2d-4i-2}$
\end{tabular}
\end{center}
and $\nu_z \phi v_i$ is a linear combination of $u_{i-1}, u_{i}, u_{i+1}$ with the following coefficients:
\begin{center}
\begin{tabular}{c|l}
{\rm term} & {\rm coefficient}\\
\hline
$u_{i-1}$ & $\gamma_i q^{2d-4i+3}(1-q^{2(i-d-1)})$\\
$u_{i}$ & $\gamma_i (q^{2d-2i+1}+q^{-2i-1}-q^{2d-4i+1}-q^{2d-4i-1})$\\
$u_{i+1}$ & $\gamma_i q^{2d-4i-3}(1-q^{2(i+1)})$
\end{tabular}
\end{center}
Using $\gamma_{i+1}/\gamma_i = q^{-1}(1 - q^{2(i+1)})$, we see that these coefficients are equal. Thus $\phi \nu_z = \nu_z \phi$. The result follows.
\end{proof}

\begin{corollary}\label{cor:unique}
For $d \in \N$, up to isomorphism there exists a unique irreducible $\A$-module of dimension $d+1$.
\end{corollary}

\begin{theorem}
Let $V$ denote an irreducible $\A$-module of dimension $d+1$.
\begin{enumerate}
\item[{\rm (i)}] If {\rm char} $\K = 2$, then $V$ extends to a unique irreducible $U_q(\mathfrak{sl}_2)$-module.
\item[{\rm (ii)}] If {\rm char} $\K \ne 2$, then $V$ extends to two non-isomorphic irreducible $U_q(\mathfrak{sl}_2)$-modules.
\end{enumerate}
\end{theorem}
\begin{proof}
By Corollary \ref{cor:unique}, $V$ is isomorphic to the $\A$-module $L(d)$ from Definition \ref{def:Ld}. Thus the $U_q(\mathfrak{sl}_2)$-modules $L(d, 1)$ and $L(d,-1)$ from Lemma \ref{lem:uqsl2modules} extend $V$. Let $W$ be an irreducible  $U_q(\mathfrak{sl}_2)$-module that extends $V$. Then $W$ has dimension $d+1$, so by Lemma \ref{lem:uqsl2modules}, $W$ is isomorphic to $L(d,1)$ or $L(d,-1)$.  Therefore, up to isomorphism, $L(d,1)$ and $L(d,-1)$ are the unique irreducible $U_q(\mathfrak{sl}_2)$-modules that extend $V$. Observe that $L(d,1)$ and $L(d,-1)$ are isomorphic as $U_q(\mathfrak{sl}_2)$-modules if and only if {\rm char} $\K = 2$, so the result follows.
\end{proof}

\pagebreak 

\section{Appendix}\label{sec:appendix}

In this appendix, we give the explicit reduction rules used in the proof of Lemma \ref{lem:spanningset}. For each forbidden $\A^\prime$-word $w$ of length 2, the reduction rule for $w$ is an equation that expresses $w$ as a linear combination of allowed $\A^\prime$-words of length 0, 1, 2. For each forbidden $\A^\prime$-word $g_1 g_2$, the linear combination given by the reduction rule contains exactly one allowed $\A$-prime-word of length 2, which we denote by $\tilde{g}_1 \tilde{g}_2$.

\

\begin{tabular}{c|l|c}
Forbidden word $g_1g_2$ & Reduction rule for $g_1g_2$ & $\tilde{g}_1 \tilde{g}_2$\\
\hline
$\nu_x^2$ & $\nu_x^2 = q^4y^2 z^2  + q^{2}(q+q^{-1})\nu_x - q^4$ & $y^2 z^2$\\
$\nu_x \nu_y$ & $\nu_x \nu_y = -q^{-1}\nu_z z^2 + q^{-2} z^2  + q^{-3} \nu_x + q\nu_y-q^{-2}$ & $\nu_z z^2$\\
$\nu_x \nu_z$ & $\nu_x \nu_z = q^2 \nu_z \nu_x + (q^2-1)y^2 - (q^2-1)$ & $\nu_z \nu_x$\\
$\nu_x x^2$ & $\nu_x x^2 = x^2 \nu_x  - (q^2-q^{-2})\nu_y + (q^2-q^{-2})\nu_z$  & $x^2 \nu_x$ \\
$\nu_x y^2$ & $\nu_x y^2 = q^4 y^2 \nu_x$ & $y^2 \nu_x$\\
\hline
$\nu_y \nu_x$ & $\nu_y \nu_x = -q \nu_z z^2 + q^2 z^2 + q^{-1} \nu_x + q^3 \nu_y - q^2$ & $\nu_z z^2$\\
$\nu_y^2$ & $\nu_y^2 = q^{-4}x^2z^2 + q^{-2}(q+q^{-1})\nu_y - q^{-4}$ & $x^2z^2$ \\
$\nu_y \nu_z$ & $\nu_y \nu_z = -q^{-1}x^2 \nu_x + q^{-2} x^2+ q \nu_y + q^{-3}\nu_z  - q^{-2}$ & $x^2 \nu_x$\\
$\nu_y x^2$ & $\nu_y x^2 = q^{-4}x^2 \nu_y$ & $x^2 \nu_y$ \\
$\nu_y y^2$ & $\nu_y y^2 = -q \nu_z \nu_x + q^{-1} y^2  + q^{-2} \nu_z + q^{2} \nu_x - q^{-1}$ & $\nu_z \nu_x$\\
\hline
$\nu_z \nu_y$ & $\nu_z \nu_y = -qx^2 \nu_x  + q^2 x^2  + q^3 \nu_y + q^{-1}\nu_z  - q^2 $ & $x^2 \nu_x$\\
$\nu_z^2$ & $\nu_z^2 = q^4 x^2 y^2 + q^2 (q+q^{-1})\nu_z - q^4$ & $x^2 y^2$\\
$\nu_z x^2$ & $\nu_z x^2 = q^4 x^2 \nu_z $ & $x^2 \nu_z$\\
\hline
$y^2 \nu_y$ & $y^2 \nu_y = -q \nu_z \nu_x + q^{-1}y^2+ q^2\nu_z + q^{-2} \nu_x  - q^{-1}$ & $\nu_z \nu_x$ \\
$y^2 \nu_z$ & $y^2 \nu_z = q^4 \nu_z y^2$ & $\nu_z y^2$\\
$y^2 x^2$ & $y^2 x^2 = q^8 x^2 y^2 + (q^6-q^{2})(q+q^{-1})\nu_z +(1 - q^8)$ & $x^2 y^2$\\
\hline
$z^2 \nu_x$ & $z^2 \nu_x = q^4 \nu_x z^2$ & $\nu_x z^2$\\
$z^2 \nu_y$ & $z^2 \nu_y = q^{-4}\nu_y z^2$ & $\nu_y z^2$\\
$z^2 \nu_z$ & $z^2 \nu_z = \nu_z z^2 + (q^2 - q^{-2})\nu_x - (q^2 - q^{-2})\nu_y$ & $\nu_z z^2$\\
$z^2 x^2$ & $z^2 x^2 = q^{-8}x^2 z^2 + (q^{-6}-q^{-2})(q+q^{-1})\nu_y + (1-q^{-8})$ & $x^2 z^2$\\
$z^2 y^2$ & $z^2 y^2 = q^8y^2 z^2  + (q^{6}-q^{2})(q+q^{-1})\nu_x + (1 - q^8)$ & $y^2 z^2$
\end{tabular}

\section{Acknowledgements}

This paper was written while the author was a graduate student at the University of Wisconsin-Madison. The author would like to thank her advisor, Paul Terwilliger, for offering many valuable ideas and suggestions.

\bibliography{Positive_even_subalgebra_of_Uqsl2}{}
\bibliographystyle{plain}
\end{document}